\documentclass[a4paper, 11pt, twoside, leqno, openright]{amsart}
\usepackage[top=1.5in, bottom=1.5in, left=0.9in, right=0.9in]{geometry}
\usepackage{amsmath}
\usepackage{epic,eepic,epsfig}
\usepackage{comment} 
\usepackage[all,cmtip]{xy}
\usepackage{mathtools}
 \usepackage[nodayofweek]{datetime}

\usepackage{amssymb}
\usepackage{amsthm}
\usepackage{amscd}
\usepackage{mathrsfs}
\usepackage{amsopn}
\usepackage{manfnt}

\usepackage{latexsym}
\usepackage{amsmath}
\usepackage{pifont}
\usepackage{lmodern}
\usepackage{array}	
\usepackage{lscape}

\usepackage{graphicx, color}
\usepackage{tikz}
\usepackage{tikz-cd}
\usepackage{pgfplots}
\usepackage {diagbox}

\usepackage{url}
\usepackage{paralist}
\usepackage{hyperref}
 
\usepackage{shuffle}

\makeatother
\newtheorem{theorem}{Theorem}[section] %Theorem (1.1)
\newtheorem{Theorem}[theorem]{Theorem}

\theoremstyle{definition}

\newtheorem{Remark}[theorem]{Remark}

\begin{document}

\def\AAAB{{\hspace{-0.5mm}\underbrace{1,\ldots,1}_{k-2~\text{times}},2}}
\def\subAAAB{{\hspace{-2mm}\underbrace{\scriptstyle 1,\ldots,1}_{k-2~\text{times}}\hspace{-2mm},2}}
\def\ellet{{\ell\text{-\'et}}}
\def\lala{\langle\!\langle}
\def\rara{\rangle\!\rangle}
\def\PP{\mbox{\large $\mathsf{P}$}}
\def\I{{\mathtt{i}}}
\def\Coeff{\mathtt{Coeff}}
\def\nyoroto{{\leadsto}}
\def\N{{\mathbb N}}
\def\C{{\mathbb C}}
\def\Z{{\mathbb Z}}
\def\R{{\mathbb R}}
\def\Q{{\mathbb Q}}
\def\bQ{{\overline{\mathbb Q}}}
\def\Gal{{\mathrm{Gal}}}
\def\et{\text{\'et}}
\def\ab{\mathrm{ab}}
\def\proP{{\text{pro-}p}}
\def\padic{{p\mathchar`-\mathrm{adic}}}
\def\la{\langle}
\def\ra{\rangle}
\def\scM{\mathscr{M}}
\def\lala{\la\!\la}
\def\rara{\ra\!\ra}
\def\ttx{{\mathtt{x}}}
\def\tty{{\mathtt{y}}}
\def\ttz{{\mathtt{z}}}
\def\bkappa{{\boldsymbol \kappa}}
\def\scLi{{\mathscr{L}i}}
\def\sLL{{\mathsf{L}}}
\def\cY{{\mathcal{Y}}}
\def\Coker{\mathrm{Coker}}
\def\Spec{\mathrm{Spec}\,}
\def\Ker{\mathrm{Ker}}
\def\CHplus{\underset{\mathsf{CH}}{\oplus}}
\def\check{{\clubsuit}}
\def\kaitobox#1#2#3{\fbox{\rule[#1]{0pt}{#2}\hspace{#3}}\ }
\def\vru{\,\vrule\,}
\newcommand*{\longhookrightarrow}{\ensuremath{\lhook\joinrel\relbar\joinrel\rightarrow}}
\newcommand{\hooklongrightarrow}{\lhook\joinrel\longrightarrow}
\def\nyoroto{{\rightsquigarrow}}
\newcommand{\pathto}[3]{#1\overset{#2}{\dashto} #3}
\newcommand{\pathtoD}[3]{#1\overset{#2}{-\dashto} #3}
\def\dashto{{\,\!\dasharrow\!\,}}
\def\ovec#1{\overrightarrow{#1}}
\def\isom{\,{\overset \sim \to  }\,}
\def\GT{{\widehat{GT}}}
\def\bfeta{{\boldsymbol \eta}}
\def\brho{{\boldsymbol \rho}}
\def\sha{\scalebox{0.6}[0.8]{\rotatebox[origin=c]{-90}{$\exists$}}}
\def\upin{\scalebox{1.0}[1.0]{\rotatebox[origin=c]{90}{$\in$}}}
\def\downin{\scalebox{1.0}[1.0]{\rotatebox[origin=c]{-90}{$\in$}}}
\def\torusA{{\epsfxsize=0.7truecm\epsfbox{torus1.eps}}}
\def\torusB{{\epsfxsize=0.5truecm\epsfbox{torus2.eps}}}
%%%%%%%%%%%%%%%%%%%%%%%%%%%%%%%%%%%%%%%%%%%%%%%%%%%%%%%5

%--------------------------------------------------------------------------------------------------------------------------
%\makeatletter
%\renewcommand{\theequation}{%
%\thesubsection.\arabic{equation}}
%\@addtoreset{equation}{subsection}
%\makeatother

\makeatletter
\renewcommand{\theequation}{%
\thesection.\arabic{equation}}
\@addtoreset{equation}{section}
\makeatother
%--------------------------------------------------------------------------------------------------------------------------
%--------------------------------------------------------------------------------------------------------------------------

\title{%[draft] 
Duality-reflection formulas of\\multiple polylogarithms
and their $\ell$-adic\\Galois analogues}
\author[D.~Shiraishi]{Densuke Shiraishi}
\address{Department of Mathematics,
Graduate School of Science,
Osaka University, Toyonaka, Osaka 560-0043, Japan}
\email{densuke.shiraishi@gmail.com,
u848765h@ecs.osaka-u.ac.jp}
\date{}

\maketitle

%%%%%%%%%%%%%%%%%%%%%%%%%%%%%%%%%%%%%%%%%%%%%%%%%%%%%%%%%%%%%%%%%%%%
% This defines the short-running title of this paper !! 
\markboth{D.Shiraishi}
{Duality-reflection formula}
%%%%%%%%%%%%%%%%%%%%%%%%%%%%%%%%%%%%%%%%%%%%%%%%%%%%%%%%%%%%%%%%%%

\begin{abstract}
In the present paper,
we derive formulas of complex and $\ell$-adic multiple polylogarithms,
which have two aspects:
a duality in terms of indexes
and a reflection in terms of variables.
We provide an algebraic proof of
these formulas
by using algebraic relations between associators 
arising from the $S_3$-symmetry of the projective line minus three points.
\end{abstract}

%\tableofcontents

%\footnote[0]{\LaTeX -compiled on: \today \ at \currenttime}

\section{Introduction and main results}

The purpose of this paper is to derive a series of functional equations that generalizes Oi-Ueno's reflection formulas between complex multiple polylogarithms at $z$ and $1-z$.
This specializes to the duality formula for multiple zeta values when $z \to 1$.
%See Remark \ref{Re1}, \ref{Re2} for details.
We also show the $\ell$-adic Galois analog of these equations by tracing the same argument in a parallel way to the complex case.

For a multi-index $\mathbf{k}=(k_1\dots,k_d) \in \mathbb{N}^d$ and a topological path
$\gamma \in \pi_1^{\rm top}\left(\mathbb{P}^1(\mathbb{C}) \backslash \{0,1,\infty\}; \overrightarrow{01}, z\right)$
from the standard tangential base point $\overrightarrow{01}$ to a (possibly, tangential base) point $z$,
the complex multiple polylogarithm 
$Li_{\mathbf{k}}\left(z;\gamma\right)$
is defined as an iterated integral along $\gamma$.
As is well known, $Li_{\mathbf{k}}\left(z;\gamma\right)$ coincides with a certain signed coefficient of the KZ fundamental solution
\[
G_0(X,Y)(z;\gamma) \in \mathbb{C} \langle \langle X,Y \rangle \rangle
\]
(See \S \ref{cmpoly}  for details).
The multiple zeta value $\zeta(\mathbf{k})$ appears as its special value at the tangential base point $\overrightarrow{10}$ with  the straight path $\delta \in \pi_1^{\rm top}\left(\mathbb{P}^1(\mathbb{C}) \backslash \{0,1,\infty\}; {\overrightarrow{01}}, \overrightarrow{10}\right)$ along the unit interval $(0,1) \subset \mathbb{P}^1(\mathbb{R}) \backslash \{0,1,\infty\}$.
Our main result of the complex case is then as follows.

\begin{Theorem}[The duality-reflection formula of complex multiple polylogarithms] \label{M1}
Given a (possibly, tangential base) point
$z$ of $\mathbb{P}^1(\mathbb{C}) \backslash \{0,1,\infty\}$
and a path $\gamma \in \pi_1^{\rm top}\left(\mathbb{P}^1(\mathbb{C}) \backslash \{0,1,\infty\}; \overrightarrow{01}, z\right)$,
define the path $\gamma'$ associated to $\gamma$ by
\begin{align}\label{gamma'}
\gamma':=\delta \cdot \phi(\gamma) \in \pi_1^{\rm top}\left(\mathbb{P}^1(\mathbb{C}) \backslash \{0,1,\infty\}; \overrightarrow{01}, 1-z\right),
\end{align}
where 
$\phi \in {\rm Aut}\left(\mathbb{P}^1(\mathbb{C}) \backslash \{0,1,\infty\}\right)$ is given by
$\phi(t)=1-t$ and
paths are composed from left to right.
For any $n,m \in \mathbb{Z}_{\geq 2}$,
the following holds:
%\begin{description}
%   \item[Complex duality-reflection formula]
\begin{align} \label{Main1}
\sum_{j=0}^{m-1}\frac{\left(-\log\left(z;\gamma \right)\right)^j}{j!}Li_{\hspace{-0.2cm}\underbrace{\scriptstyle 1,\ldots,1}_{n-2~\text{times}}\hspace{-0.2cm},m-j}~\hspace{-0.2cm}\left(z;\gamma\right)
&+\sum_{j=0}^{n-2}\frac{\left(-\log\left(1-z;\gamma'\right)\right)^j}{j!}Li_{\hspace{-0.2cm}\underbrace{\scriptstyle 1,\ldots,1}_{m-2~\text{times}}\hspace{-0.2cm},n-j}~\hspace{-0.2cm}\left(1-z;\gamma'\right)\\
&=\zeta({\underbrace{1,\ldots,1}_{n-2~\text{times}},m}) \notag
\end{align}
%\end{description}
where $\log(z;\gamma):=\int_{\delta^{-1} \cdot \gamma}\frac{dt}{t}$
%, $\log(1-z;\gamma')=-\int_{\gamma}\frac{dt}{1-t}$
is the logarithm function with respect to $\gamma$.
\end{Theorem}
\noindent
This functional equation has two aspects:
a duality $n \leftrightarrow m$ with respect to indexes
and a reflection $z \leftrightarrow 1-z$ with respect to variables.
We derive the functional equation from an algebraic relation (chain rule) between associators
\[
G_0(X,Y)(z;\gamma)=G_0(Y,X)(1-z;\gamma') \cdot G_0(X,Y)(\overrightarrow{10};\delta)
\]
along the path composition (\ref{gamma'}),
where $G_0(X,Y)(\overrightarrow{10};\delta)$ is the so-called Drinfeld associator.

We also deal with the $\ell$-adic Galois case for any prime number $\ell$.
Let $K$ be a subfield of $\mathbb{C}$ and
$G_K:=\Gal(\overline{K}/K)$
the absolute Galois group of $K$ with respect to its algebraic closure $\overline{K}$ in $\mathbb{C}$. 
Suppose $z$ is $K$-rational (possibly, tangential base) point of $\mathbb{P}^1 \backslash \{0,1,\infty\}$.
Consider each $\gamma \in \pi_1^{\rm top}\left(\mathbb{P}^1(\mathbb{C}) \backslash \{0,1,\infty\}; \overrightarrow{01}, z\right)$ as a pro-$\ell$ \'etale path $\gamma \in \pi_1^\ellet \left(\mathbb{P}^1_{\overline{K}} \backslash \{0,1,\infty\};\overrightarrow{01},{z}\right)$
by the comparison map.
For $\sigma \in G_K$,
the $\ell$-adic Galois multiple polylogarithm ${Li}^\ell_{\mathbf{k}}(z;\gamma,\sigma)$
is defined as a certain signed coefficient of the $\ell$-adic Galois associator 
\[
{\mathfrak f}^{z,\gamma}_{\sigma}(X,
Y) \in \mathbb{Q}_{\ell} \langle \langle X,
Y \rangle \rangle.
\]
This $\ell$-adic multiple polylogarithm is an $\ell$-adic \'etale avatar of $Li_{\mathbf{k}}\left(z;\gamma\right)$ introduced by Wojtkowiak (See \S \ref{lmpoly}  for details).
The $\ell$-adic Galois multiple zeta value (or called $\ell$-adic multiple Soul\'e element) ${\boldsymbol \zeta}_{{\mathbf k}}^\ell(\sigma)$ is defined as its special value ${Li}^\ell_{\mathbf{k}}\left(\overrightarrow{10};\delta,\sigma\right)$.
Our another main result is then as follows.

\begin{Theorem}[The duality-reflection formula of $\ell$-adic Galois multiple polylogarithms] \label{M2}
Given a $K$-rational (possibly, tangential base) point
$z$ of $\mathbb{P}^1 \backslash \{0,1,\infty\}$
and $\gamma \in \pi_1^{\rm top}\left(\mathbb{P}^1(\mathbb{C}) \backslash \{0,1,\infty\}; \overrightarrow{01}, z\right)$,
define the path $\gamma'$ associated to $\gamma$ as in (\ref{gamma'}).
For any $\sigma \in G_K$,
the following holds:
%\begin{description}
%   \item[$\ell$-adic duality-reflection formula]
\begin{align} \label{Main2}
\sum_{j=0}^{m-1}\frac{\left(\rho_{z,\gamma}(\sigma)\right)^j}{j!}Li^{\ell}_{\hspace{-0.2cm}\underbrace{\scriptstyle 1,\ldots,1}_{n-2~\text{times}}\hspace{-0.2cm},m-j}~\hspace{-0.1cm}(z;\gamma,\sigma)
&+\sum_{j=0}^{n-2}\frac{\left(\rho_{1-z,\gamma'}(\sigma)\right)^j}{j!}Li^{\ell}_{\hspace{-0.2cm}\underbrace{\scriptstyle 1,\ldots,1}_{m-2~\text{times}}\hspace{-0.2cm},n-j}~\hspace{-0.1cm}(1-z;\gamma',\sigma)\\
&={\boldsymbol \zeta}^{\ell}_{\hspace{-0.2cm}\underbrace{\scriptstyle 1,\ldots,1}_{n-2~\text{times}}\hspace{-0.2cm},m}(\sigma) \notag
\end{align}
%\end{description}
where
$\rho_{z,\gamma}: G_K \to \mathbb{Z}_{\ell}$ is the Kummer 1-cocycle defined by $\sigma(z^{1/{\ell^k}})=\zeta_{\ell^k}^{\rho_{z,\gamma}(\sigma)~{\rm mod}~\ell^k}z^{1/{\ell^k}}$
with respect to the $\ell$-th power roots $\{z^{1/{\ell^k}}\}_{k}$ determined by $\gamma$.
\end{Theorem}
\noindent
By reinterpreting the proof of the complex case (\ref{Main1}) after replacing $G_0(X,Y)(z;\gamma)$ by ${\mathfrak f}^{z,\gamma}_{\sigma}(X,
Y)$,
we derive
the $\ell$-adic functional equation (\ref{Main2}) from an chain rule between $\ell$-adic Galois associators
\[
{\mathfrak f}^{z,\gamma}_{\sigma}(X,
Y)={\mathfrak f}^{1-z,\gamma'}_{\sigma}(Y,
X) \cdot {\mathfrak f}^{\overrightarrow{10},\delta}_{\sigma}(X,
Y)
\]
along the path composition (\ref{gamma'}).

\begin{Remark}\label{Re1}
The formula (\ref{Main1}) is a generalization of the following functional equation (\ref{Oi1}) due to Oi and Ueno in \cite{Oi09},\cite{OU13},
and (\ref{Main2}) is a generalization of the following functional equation (\ref{Nakamura}) due to Nakamura in \cite{NS22},\cite{N21}.
\begin{align} \label{Oi1}
& \sum_{j=0}^{m-1}\frac{\left(-\log\left(z;\gamma\right)\right)^j}{j!}Li_{m-j}\left(z;\gamma\right)+Li_{\hspace{-0.2cm}\underbrace{\scriptstyle 1,\ldots,1}_{m-2~\text{times}}\hspace{-0.2cm},2}\left(1-z;\gamma'\right)=\zeta(m),\\
%\quad [{\rm Oi\text{-}Ueno}, 2013],[{\rm Oi}, 2009]
%\tag{$\#_{\mathbb{C}.1}$},
%\label{OI2}
%& \sum_{j=0}^{k-1}\frac{(-\log(z))^{j}}{j!}Li_{1,k-j}(z)+Li_{\underbrace{\scriptstyle 1,\ldots,1}_{k-2~\text{times}},3}~(1-z)-\log(1-z)Li_{\underbrace{\scriptstyle 1,\ldots,1}_{k-2~\text{times}},2}~(1-z) =\zeta(1,k),\\
%\tag{$\#_{\mathbb{C}.2}$}
%\quad [{\rm Oi}, 2009].
\label{Nakamura} 
& \sum_{j=0}^{m-1}\frac{\left(\rho_{z,\gamma}(\sigma)\right)^j}{j!}Li^{\ell}_{m-j}\left(z;\gamma,\sigma\right)+Li^{\ell}_{\hspace{-0.2cm}\underbrace{\scriptstyle 1,\ldots,1}_{m-2~\text{times}}\hspace{-0.2cm},2}\left(1-z;\gamma',\sigma\right)={\boldsymbol \zeta}^{\ell}_{m}(\sigma) \quad (\sigma \in G_K).
\end{align}
\end{Remark}

\begin{Remark}\label{Re2}
By setting $z=\overrightarrow{10}$ (i.e. $z \to 1$ along the real interval) in (\ref{Main1}) and (\ref{Main2}),
we obtain the well-known duality formula of multiple zeta values and its $\ell$-adic Galois analog.
\begin{align} \label{c-dual} 
\zeta({\underbrace{1,\ldots,1}_{m-2~\text{times}},n})&=\zeta({\underbrace{1,\ldots,1}_{n-2~\text{times}},m}),\\
{\boldsymbol \zeta}^{\ell}_{\hspace{-0.2cm}\underbrace{\scriptstyle 1,\ldots,1}_{m-2~\text{times}}\hspace{-0.2cm},n}(\sigma)&={\boldsymbol \zeta}^{\ell}_{\hspace{-0.2cm}\underbrace{\scriptstyle 1,\ldots,1}_{n-2~\text{times}}\hspace{-0.2cm},m}(\sigma) \quad (\sigma \in G_K).
\end{align}
\end{Remark}

\begin{Remark}
In \cite{F04},
Furusho constructed the theory of the $p$-adic KZ equation
and studied the $p$-adic multiple polylogarithm,
which is a $p$-adic crystalline avatar of $Li_{\mathbf{k}}\left(z;\gamma\right)$.
Using the results in \cite{F04},
it is possible to obtain a $p$-adic analog of (\ref{Main1}) in the same way as in the proof of (\ref{Main1}).
\end{Remark}

\bigskip
\noindent
{\it Acknowledgement.}~
The author would like to express deep gratitude to Professor Hiroaki Nakamura for his helpful advice and warm encouragement.
This work was supported by JSPS KAKENHI Grant Numbers JP20J11018.

\section{Preliminaries}

In this section,
we review the basic properties of complex multiple polylogarithms and $\ell$-adic  Galois multiple polylogarithms in preparation for proving the main theorems (\ref{Main1}) and (\ref{Main2}).

For a (possibly, tangential base) point $z$ on $\mathbb{P}^1 \backslash \{0,1,\infty\}$,
we shall write
\[
\pi_1^{\rm top}\left(\mathbb{P}^1(\mathbb{C}) \backslash \{0,1,\infty\}; \overrightarrow{01}, z\right)
\]
for the set of  homotopy classes of piece-wise smooth topological paths on
$\mathbb{P}^1(\mathbb{C}) \backslash \{0,1,\infty\}$ from the tangential base point $\overrightarrow{01}$ to $z$,
and
\[
\pi_1^{\rm top}\left(\mathbb{P}^1(\mathbb{C}) \backslash \{0,1,\infty\}, \overrightarrow{01}\right):=\pi_1^{\rm top}\left(\mathbb{P}^1(\mathbb{C}) \backslash \{0,1,\infty\}; \overrightarrow{01}, \overrightarrow{01}\right)
\]
for the topological fundamental group of $\mathbb{P}^1(\mathbb{C}) \backslash \{0,1,\infty\}$ at the base point $\overrightarrow{01}$
with respect to the path composition $\gamma_1 \cdot \gamma_2:=\gamma_1 \gamma_2$ i.e. paths are composed from left to right.
Let 
\[
l_0,
l_1 \in \pi_1^{\rm top}\left(\mathbb{P}^1(\mathbb{C}) \backslash \{0,1,\infty\}, \overrightarrow{01}\right)
\]
be homotopy classes of smooth loops
circling counterclockwise around
$0$, $1$ respectively,
as FIGURE \ref{paths} shows.
In FIGURE \ref{paths}, the dashed line represents ${\mathbb P}^1({\mathbb R})\backslash \{0,1,\infty\}$ and the upper half-plane is located above.
Then,
$\{l_0,l_1\}$ is a free generating system of $\pi_1^{\rm top}\left(\mathbb{P}^1(\mathbb{C}) \backslash \{0,1,\infty\}, \overrightarrow{01}\right)$.
Fix a homotopy class
\begin{align} \label{pathz}
\gamma \in \pi_1^{\rm top}\left(\mathbb{P}^1(\mathbb{C}) \backslash \{0,1,\infty\}; \overrightarrow{01}, z\right).
\end{align}
Moreover,
we denote by
\[
\delta \in \pi_1^{\rm top}\left(\mathbb{P}^1(\mathbb{C}) \backslash \{0,1,\infty\}; {\overrightarrow{01}}, \overrightarrow{10}\right)
\]
the homotopy class of  the smooth path on $\mathbb{P}^1(\mathbb{C}) \backslash \{0,1,\infty\}$ from $\overrightarrow{01}$ to $\overrightarrow{10}$ along the real interval
as FIGURE \ref{paths} shows.
Let $\phi \in {\rm Aut}\left(\mathbb{P}^1(\mathbb{C}) \backslash \{0,1,\infty\}\right)$ be the automorphism of 
${\mathbb P}^1(\mathbb{C}) \backslash \{0,1,\infty\}$ defined by
$\phi(t)=1-t$.
Then,
we shall define
\begin{align} \label{path1-z}
\gamma':=\delta \cdot \phi(\gamma) \in \pi_1^{\rm top}\left(\mathbb{P}^1(\mathbb{C}) \backslash \{0,1,\infty\}; \overrightarrow{01}, 1-z\right).
\end{align}

\begin{center}
\begin{figure}[hbtp]
\caption{Topological paths on ${\mathbb P}^1({\mathbb C})\backslash 
\{0,1,\infty\}$}
\label{paths}
\begin{tikzpicture} \label{picture}
\draw (2,0) -- (6,0);
\draw (4.6,0) -- (4.5,0.1);
\draw (4.6,0) -- (4.5,-0.1);

\draw (0.1,0) -- (0,0.1);
\draw (0.1,0) -- (0.2,0.1);

\draw (7.3,0) -- (7.4,-0.1);
\draw (7.3,0) -- (7.2,-0.1);

%\draw (4,-1) -- (3.9,-0.9);
%\draw (4,-1) -- (4.1,-0.9);

\draw (2.1,0) to [out=0,in=270] (2.8,1);
\draw (2.8,1) to [out=90,in=-10] (2,1.9);
\draw (2,1.9) to [out=170,in=90] (0.1,0);
\draw (0.1,0) to [out=270,in=190] (2,-1.9);
\draw (2,-1.9) to [out=10,in=270] (2.8,-0.7);

\draw (2.8,-0.7) to [out=90,in=360] (2.1,0);
\draw (2,0) to [out=0,in=210] (4,0.4);
\draw (4,0.4) to [out=30,in=180] (6,1.4);
\draw (6,1.4) to [out=0,in=90] (7.3,0);
\draw (7.3,0) to [out=270,in=0] (6,-1.4);
\draw (4,-0.4) to [out=330,in=180] (6,-1.4);
\draw (4,-0.4) to [out=150,in=0] (2,0);

\draw (2.1,0) to [out=0,in=180] (5,1.9);

\draw (6,0) to [out=180,in=0] (3,-1.9);

%\draw (2.1,0) to [out=0,in=90] (4,-1);
%\draw (4,-1) to [out=-90,in=0] (3,-1.9);

\draw (3.67,1.1) -- (3.48,1.1);
\draw (3.67,1.1) -- (3.67,0.9);

\draw (3.7,-1.71) -- (3.75,-1.5);
\draw (3.7,-1.71) -- (3.9,-1.75);

\node at (3.6,1.6) {$\gamma$};
\node at (4.5,-1.8) {$\phi(\gamma)$};
\node at (-0.2,0) {$l_0$};
\node at (1.8,-0.4) {$0$};
\node at (4.8,0.35) {$\delta$};
\node at (6.2,-0.4) {$1$};
\node at (7.7,0) {$l_1$};
\node at (2,0) {${\bullet}$};
\node at (6,0) {${\bullet}$};
\node at (5,1.9) {${\bullet}$};
\node at (5,2.2) {${z}$};
\node at (3,-1.9) {${\bullet}$};
\node at (3,-2.2) {${\scriptstyle 1-z}$};

%\node at (3.75,-2.8) {The dashed line represents ${\mathbb P}^1({\mathbb R})\backslash 
%\{0,1,\infty\}$.};
%\node at (3.75,-3.3) {Above the dashed line is the upper half-plane.};

\draw[dotted](-1,0)--(8.5,0);
\end{tikzpicture}
\end{figure}
\end{center}

\subsection{Complex multiple polylogarithms}\label{cmpoly}
Let $z$ be a $\mathbb{C}$-rational (possibly, tangential base) point on $\mathbb{P}^1 \backslash \{0,1,\infty\}$.
For a pair $\mathbf{k}=(k_1\dots,k_d) \in \mathbb{N}^d$ and a path $\gamma (=\gamma_z) \in \pi_1^{\rm top}\left(\mathbb{P}^1(\mathbb{C}) \backslash \{0,1,\infty\}; \overrightarrow{01}, z\right)$,
we shall define the complex logarithm
\begin{align} \label{clog}
\log(z;\gamma):=-\int_{\delta^{-1} \cdot \gamma}\frac{dt}{t}
\end{align}
and the complex multiple polylogarithm $Li_{\mathbf{k}}(z;\gamma)$ as the iterated integral along $\gamma$ below:
\begin{align} \label{cpoly1}
& {Li}_{\mathbf{k}}\left(z;\gamma\right):=
\begin{cases}
\int_{\gamma} \frac{1}{t}Li_{k_1,\cdots,k_d-1}\left(t;\gamma_{t}\right)dt & k_d\neq 1, \\
\int_{\gamma} \frac{1}{1-t}Li_{k_1,\cdots,k_{d-1}}\left(t;\gamma_{t}\right)dt & k_d= 1, \\
\end{cases}  \\ \label{cpoly2}
& Li_1\left(z;\gamma\right):=-\log(1-z;\gamma)=\int_{\gamma}\frac{dt}{1-t},
\end{align}
which can be analytically continued to the pointed universal covering space of
$\left(\mathbb{P}^1(\mathbb{C}) \backslash \{0,1,\infty\},\overrightarrow{01}\right)$.
In particular,
we define 
the multiple zeta value
\begin{align} \label{zeta}
\zeta(\mathbf{k}):={Li}_{\mathbf{k}}\left(\overrightarrow{10};\delta\right) \in \mathbb{R}.
\end{align}

The complex multiple polylogarithm $Li_{\mathbf{k}}(z;\gamma)$ is closely related to the KZ (Knizhnik-Zamolodchikov) equation.
The formal KZ equation on $\mathbb{P}^1(\mathbb{C}) \backslash \{0,1,\infty\}$ is the differential equation
\[
\dfrac{d}{d z}G(X,Y)(z)=\left( \dfrac{X}{z}+\dfrac{Y}{z-1} \right)G(X,Y)(z)
\]
where
$G(X, Y)(z)$ is an analytic (i.e. each of whose coefficient is analytic) function with values in the non-commutative formal power series algebra $\mathbb{C} \langle \langle X, Y \rangle \rangle$.
There exists a unique solution $G_0(X,Y)(z;\gamma) \in \mathbb{C} \langle \langle X,Y \rangle \rangle$ attached to $\gamma \in \pi_1^{\rm top}\left(\mathbb{P}^1(\mathbb{C}) \backslash \{0,1,\infty\}; \overrightarrow{01}, z\right)$ characterized by the asymptotic behavior
\begin{align} \label{G0}
G_0(X,Y)(z;\gamma) \approx \sum_{m=0}^{\infty}\frac{1}{m!}{\left(X \cdot \log \left(z;\gamma\right)\right)}^m~(z \to 0).
\end{align}
Moreover,
we define the Drinfeld associator
\[
\Phi(X,Y):=G_0(X,Y)\left(\overrightarrow{10};\delta\right) \in {\mathbb{C} \langle \langle X,Y \rangle \rangle}.
\]
Then
the following relation (chain rule) holds:
\begin{align} \label{rel1}
& G_0(X,Y)(z;\gamma)=G_0(Y,X)(1-z;\gamma') \cdot \Phi(X,Y).
\end{align}
This relation reflects the path composition (\ref{path1-z}).
Let ${\rm M}$ be the non-commutative free monoid generated by the non-commuting indeterminates $X$, $Y$.
Since $G_0(X,Y)(z;\gamma)$ is group-like in $\mathbb{C} \langle \langle X,Y \rangle \rangle$,
the expansion of $G_0(X,Y)(z;\gamma)$ looks like
\begin{equation}
G_0(X,Y)(z;\gamma)=1+\sum_{w \in {\rm M} \backslash \{1\}}
\Coeff_{w}(G_0(X,Y)(z;\gamma)) \cdot w
\end{equation}
where $\{\Coeff_{w}(G_0(X,Y)(z;\gamma))\}_{w \in {\rm M}}$ is a family of complex numbers.
For $w(\mathbf{k}):=X^{k_d-1}Y\cdots X^{k_1-1}Y$ and 
$j \in \mathbb{N}$,
the following holds (cf. \cite{F04},\cite{F14},\cite{LM96}):
\begin{align} \label{rel2}
& \Coeff_{w(\mathbf{k})}(G_0(X,Y)(z;\gamma))=(-1)^{d} \cdot Li_{\mathbf{k}}(z;\gamma),\\
\label{rel_log}
& \Coeff_{X^j}(G_0(X,Y)(z;\gamma))=\dfrac{\log^{j}(z;\gamma)}{j!},\\
\label{rel_log2}
&\Coeff_{X^j}(G_0(X,Y)(1-z;\gamma'))=\dfrac{\log^{j}(1-z;\gamma')}{j!}.
\end{align}

\subsection{$\ell$-adic Galois multiple polylogarithms}\label{lmpoly}
Let $\ell$ be a prime number and
$K$ a subfield of $\mathbb{C}$ with the algebraic closure $\overline{K} \subset \mathbb{C}$.
Suppose that $z$ is $K$-rational (possibly, tangential base) point on $\mathbb{P}^1 \backslash \{0,1,\infty\}$.
Then
the $\ell$-adic Galois (multiple) polylogarithm introdecud by Zdzis{\l}aw Wojtkowiak in his series of papers \cite{W0}-\cite{W3} is defined as follows.

We shall write
\[
\pi_1^\ellet\left(\mathbb{P}^1_{\overline{K}} \backslash \{0,1,\infty\};\overrightarrow{01},{z}\right)
\]
for the pro-$\ell$-finite set of pro-$\ell$ \'etale paths
 on $\mathbb{P}^1_{\overline{K}} \backslash \{0,1,\infty\}$ from the $K$-rational tangential base point $\overrightarrow{01}$ to ${z}$,
and
\[
\pi_1^\ellet\left(\mathbb{P}^1_{\overline{K}} \backslash \{0,1,\infty\},\overrightarrow{01}\right):=\pi_1^\ellet\left(\mathbb{P}^1_{\overline{K}} \backslash \{0,1,\infty\};\overrightarrow{01},\overrightarrow{01}\right)
\]
for the pro-$\ell$ \'etale fundamental group of $\mathbb{P}^1_{\overline{K}} \backslash \{0,1,\infty\}$ with the base point $\overrightarrow{01}$.
By the comparison map
\begin{align} \label{compmap}
\pi_1^{\rm top}\left(\mathbb{P}^1(\mathbb{C}) \backslash \{0,1,\infty\}; \overrightarrow{01}, \ast\right) \to \pi_1^\ellet\left(\mathbb{P}^1_{\overline{K}} \backslash \{0,1,\infty\};\overrightarrow{01},{\ast}\right)
\end{align} 
{for} $\ast \in \left\{\overrightarrow{01},z,1-z\right\}$,
we regard topological paths $l_0,l_1,\gamma,\gamma'$ on $\mathbb{P}^1(\mathbb{C}) \backslash \{0,1,\infty\}$ as pro-$\ell$ \'etale paths on $\mathbb{P}^1_{\overline{K}} \backslash \{0,1,\infty\}$.
Then $\pi_1^\ellet\left(\mathbb{P}^1_{\overline{K}} \backslash \{0,1,\infty\},\overrightarrow{01}\right)$ is the free pro-$\ell$ group of rank $2$ with topologically generating system $\{l_0,l_1\}$.

We focus on the natural action of $G_K$ on $\pi_1^\ellet \left(\mathbb{P}^1_{\overline{K}} \backslash \{0,1,\infty\};\overrightarrow{01},{z}\right)$ (cf. \cite[2.8]{N99}, \cite[(1.1)]{NW99}).
Since $z$ is $K$-rational,
this Galois action is well-defined.
For each $\sigma \in G_K$,
we define a pro-$\ell$ \'etale loop
\begin{align} \label{f}
{\mathfrak f}^{z,\gamma}_{\sigma}:=
\gamma \cdot \sigma(\gamma)^{-1} \in \pi_1^\ellet\left(\mathbb{P}^1_{\overline{K}} \backslash \{0,1,\infty\},\overrightarrow{01}\right).
\end{align}
Consider the multiplicative Magnus embedding into the algebra of non-commutative formal power series
\[E :\pi_1^\ellet\left(\mathbb{P}^1_{\overline{K}} \backslash \{0,1,\infty\},\overrightarrow{01}\right) \hookrightarrow \mathbb{Q}_{\ell} \langle \langle X,
Y \rangle \rangle\]
defined by
$E(l_0)={\rm exp}(X):=\sum_{n=0}^{\infty}\frac{1}{n!}{X}^n,~
E(l_1)={\rm exp}(Y)$.
We get a formal power series
\begin{align} \label{f2}
{\mathfrak f}^{z,\gamma}_{\sigma}(X,
Y):=E({\mathfrak f}^{z,\gamma}_{\sigma}) \in \mathbb{Q}_{\ell} \langle \langle X,
Y \rangle \rangle
\end{align}
called the $\ell$-adic Galois associator.
If $z=\overrightarrow{10}$ and $\gamma=\delta$, 
it is called the $\ell$-adic Ihara associator in \cite[Definition 2.32]{F07}.
By the path composition (\ref{path1-z}) and (\ref{f}),
the following relation holds:
\begin{align} \label{rel3}
& {\mathfrak f}^{z,\gamma}_{\sigma}(X,
Y)={\mathfrak f}^{1-z,\gamma'}_{\sigma}(Y,
X) \cdot {\mathfrak f}^{\overrightarrow{10},\delta}_{\sigma}(X,
Y).
\end{align}
The power series ${\mathfrak f}^{z,\gamma}_{\sigma}(X,
Y)$ is an $\ell$-adic Galois analog of the KZ fundamental solution $G_0(X,Y)(z;\gamma)$ in (\ref{G0}), and
the relation (\ref{rel3}) is an $\ell$-adic Galois analog of the chain rule (\ref{rel1}) of KZ fundamental solutions.
Since
${\mathfrak f}^{z,\gamma}_{\sigma}(X,
Y)$ is group-like in $\mathbb{Q}_{\ell} \langle \langle X,
Y \rangle \rangle$,
the expansion of ${\mathfrak f}^{z,\gamma}_{\sigma}(X,
Y)$ looks like
\begin{equation}
{\mathfrak f}^{z,\gamma}_{\sigma}(X,
Y)=1+\sum_{w \in {\rm M} \backslash \{1\}}
\Coeff_{w}\left({\mathfrak f}^{z,\gamma}_{\sigma}(X,
Y)\right) \cdot w,
\end{equation}
where $\{\Coeff_{w}({\mathfrak f}^{z,\gamma}_{\sigma}(X,
Y))\}_{w \in {\rm M}}$ is a family of $\ell$-adic numbers.
For $\mathbf{k}=(k_1\dots,k_d) \in \mathbb{N}^d$ and
$w(\mathbf{k}):=X^{k_d-1}Y\cdots X^{k_1-1}Y$,
we shall define the $\ell$-adic Galois multiple polylogarithm and the $\ell$-adic Galois multiple zeta value
\begin{align} \label{lpoly1}
{Li}^\ell_{\mathbf{k}}(z;\gamma,\sigma) 
&:=(-1)^{d} \cdot \Coeff_{w(\mathbf{k})}({\mathfrak f}^{z,\gamma}_{\sigma}(X,Y)), \\
\label{lpoly2}
{\boldsymbol \zeta}_{{\mathbf k}}^\ell(\sigma)&:={Li}^\ell_{\mathbf{k}}\left(\overrightarrow{10};\delta,\sigma\right).
\end{align}
As ${\boldsymbol \zeta}_{{\mathbf k}}^\ell(\sigma)$ is called the $\ell$-adic multiple Soul\'e element in \cite[Definition 2.32]{F07},
${\boldsymbol \zeta}_{{k}}^\ell(\sigma)$ is closely related to the Soul\'e character (cf. \cite[Examples 2.33]{F07}, \cite[REMARK 2]{NW99}).

Let 
$\rho_{z,\gamma}~(\text{resp.}~\rho_{1-z,\gamma'}):G_K \to \Z_\ell$
be the Kummer 1-cocycle of $\{z^{1/{\ell^k}}\}_{k}$ (resp. $\{(1-z)^{1/{\ell^k}}\}_{k}$) determined by $\gamma$ (resp. $\gamma'$).
For $j \in \mathbb{N}$,
the following holds (cf. \cite{NW99},\cite{NW20},\cite{NS22}):
\begin{align} \label{rel_rho}
& \Coeff_{X^j}({\mathfrak f}^{z,\gamma}_{\sigma}(X,
Y))=\dfrac{(-\rho_{z,\gamma}(\sigma))^j}{j!},\\
\label{rel_rho2}
&\Coeff_{X^j}({\mathfrak f}^{1-z,\gamma'}_{\sigma}(X,
Y))=\dfrac{(-\rho_{1-z,\gamma'}(\sigma))^j}{j!}.
\end{align}

The $\ell$-adic Galois multiple polylogarithm is similar to the complex multiple polylogarithm as the TABLE \ref{table} shows.

\renewcommand{\arraystretch}{2.0}
\tabcolsep = 0.4cm
\begin{table}[htb]
\centering
  \caption{Analogy between $\ell$-adic Galois side and complex side}
  \label{table}
  \begin{tabular}{|c||c|}  \hline 
    $\ell$-adic Galois side & complex side  \\ \hline \hline

$z$ : $K$-ratinal base point on $\mathbb{P}^1 \backslash \{0,1,\infty\}$ & $z$ : $\mathbb{C}$-rational base point on $\mathbb{P}^1 \backslash \{0,1,\infty\}$ \\ \hline

    ${\mathfrak f}^{z,\gamma}_{\sigma}(X,
Y) \in \mathbb{Q}_{\ell} \langle \langle X,
Y \rangle \rangle~(\sigma \in G_K)$ & $G_0(X,Y)(z;\gamma) \in \mathbb{C} \langle \langle X,Y \rangle \rangle$ \\ \hline

     ${\mathfrak f}^{\overrightarrow{10},\delta}_{\sigma}(X,
Y) \in \mathbb{Q}_{\ell} \langle \langle X,
Y \rangle \rangle$ & $\Phi(X,Y)=G_0(X,Y)\left(\overrightarrow{10};\delta\right) \in \mathbb{C} \langle \langle X,
Y \rangle \rangle$  \\ \hline

  ${\mathfrak f}^{z,\gamma}_{\sigma}(X,
Y)={\mathfrak f}^{1-z,\gamma'}_{\sigma}(Y,
X) \cdot {\mathfrak f}^{\overrightarrow{10},\delta}_{\sigma}(X,
Y)$  &  $G_0(X,Y)(z;\gamma)=G_0(Y,X)(1-z;\gamma') \cdot \Phi(X,Y)$ \\ \hline

  ${Li}^{\ell}_{\mathbf{k}}(z;\gamma,\sigma) \in  \mathbb{Q}_{\ell}$ & $Li_{\mathbf{k}}(z;\gamma) \in  \mathbb{C}$ \\ \hline 
   ${\boldsymbol \zeta}_{{\mathbf k}}^{\hspace{0.04cm}\ell}: G_K \to \mathbb{Q}_{\ell}$ & $\zeta(\mathbf{k}) \in \mathbb{R}$ \\ \hline 

${Li}^{\ell}_{1}(z;\gamma,\sigma)=\rho_{1-z,\gamma'}(\sigma)$ & $Li_{1}(z;\gamma)=-\log(1-z,\gamma')$ \\ \hline 
  \end{tabular}
\end{table}

\section{Proof of main results}
In this section,
we prove Theorem \ref{M1} and Theorem \ref{M2}.
We fix a topological path $\gamma \in \pi_1^{\rm top}\left(\mathbb{P}^1(\mathbb{C}) \backslash \{0,1,\infty\}; \overrightarrow{01}, z\right)$.
All other symbols are the same as in the previous sections.

\begin{proof}[Proof of Theorem \ref{M1}, Theorem \ref{M2}]
Let $n,m \in \mathbb{Z}_{\geq 2}$.
The following computations are inspired by a remark given in the Appendix of Furusho's lecture note \cite[A.24]{F14} and an insight about the $\ell$-adic Oi-Ueno's equation in Nakamura's Oberwolfach Report \cite{N21}.

First,
we show Theorem \ref{M1}.
%Let $z$ be a $\mathbb{C}$-rational (possibly, tangential base) point on $\mathbb{P}^1 \backslash \{0,1,\infty\}$.
Since $G_0(X,Y)(z;\gamma)$ is group-like in $\mathbb{C} \langle \langle X,Y \rangle \rangle$,
the shuffle relation holds for $\{\Coeff_{w}(G_0(X,Y)(z;\gamma))\}_{w \in {\rm M}}$ (cf. [Ree58]), i.e. for $w, w' \in {\rm M}$,
\begin{align} \label{shuffle relation}
\Coeff_{w \shuffle w'}=\Coeff_{w} \cdot \Coeff_{w'}.
\end{align}
By the definition of the shuffle product,
\begin{align} \label{sh1}
X^j \shuffle X^{m-j-1} Y^{n-1}=X(X^{j-1} \shuffle X^{m-j-1} Y^{n-1})+X(X^{j} \shuffle X^{m-j-2} Y^{n-1}).
\end{align}
For $w, w' \in {\rm M}$,
we set
\[
\Coeff_{w+w'}:=\Coeff_{w}+\Coeff_{w'}.
\]
Then,
we obtain
\begin{align*}
&\sum_{j=0}^{m-1}\frac{\left(-\log(z;\gamma)\right)^j}{j!}Li_{\hspace{-0.2cm}\underbrace{\scriptstyle 1,\ldots,1}_{n-2~\text{times}}\hspace{-0.2cm},m-j}(z)\\
=&\sum_{j=0}^{m-1} (-1)^{n+j-1} \cdot \Coeff_{X^j}\Bigl(G_0(Y,X)(z;\gamma)\Bigr) \cdot \Coeff_{X^{m-j-1} Y^{n-1}}\Bigl(G_0(X,Y)(z;\gamma)\Bigr) \quad ({\rm by}~(\ref{rel_log}),(\ref{rel2}))\\
=&\sum_{j=0}^{m-1} (-1)^{n+j-1} \cdot \Coeff_{X^j \shuffle X^{m-j-1} Y^{n-1}}\Bigl(G_0(X,Y)(z;\gamma)\Bigr) \quad \left(\text{by}~\left(\ref{shuffle relation}\right)\right)\\
=&(-1)^{n+m-2} \cdot \Coeff_{Y(X^{m-1} \shuffle Y^{n-2})}\Bigl(G_0(X,Y)(z;\gamma)\Bigr) \quad ({\rm by~(\ref{sh1})}).
\end{align*}
Using (\ref{rel_log2}), (\ref{rel2}), (\ref{zeta}), $\log\left(\overrightarrow{10};\delta\right)=0$,
(\ref{shuffle relation}) and (\ref{sh1}),
we have the following equalities 
by making the same computations as above: 
\begin{align*}
\sum_{j=0}^{n-2}\frac{(-\log(1-z;\gamma'))^j}{j!}Li_{\hspace{-0.2cm}\underbrace{\scriptstyle 1,\ldots,1}_{m-2~\text{times}}\hspace{-0.2cm},n-j}(1-z)
=(-1)^{n+m-3} \cdot \Coeff_{X(Y^{m-1} \shuffle X^{n-2})}\Bigl(G_0(X,Y)(1-z;\gamma')\Bigr),
\end{align*}
and
\begin{align*}
\zeta({\underbrace{1,\ldots,1}_{n-2~\text{times}},m})
=&\left(Li_{\hspace{-0.2cm}\underbrace{\scriptstyle 1,\ldots,1}_{n-2~\text{times}}\hspace{-0.2cm},m}\left(\overrightarrow{10};\delta\right) + \sum_{j=1}^{m-1}\frac{\left(-\log\left(\overrightarrow{10};\delta\right)\right)^j}{j!}Li_{\hspace{-0.2cm}\underbrace{\scriptstyle 1,\ldots,1}_{n-2~\text{times}}\hspace{-0.2cm},m-j}\left(\overrightarrow{10};\delta\right)\right) \\
=&(-1)^{n+m-2} \cdot \Coeff_{Y(X^{m-1} \shuffle Y^{n-2})}\Bigl(G_0(Y,X)\left(\overrightarrow{10};\delta\right)\Bigr).
\end{align*}
Combining 
these equalities
and the following equality
\begin{align*}
&\Coeff_{Y(X^{m-1} \shuffle Y^{n-2})}\Bigl(G_0(X,Y)(z;\gamma)\Bigr)\\
=&\Coeff_{Y(X^{m-1} \shuffle Y^{n-2})}\Bigl(G_0(Y,X)(1-z;\gamma')\Bigr)+\Coeff_{Y(X^{m-1} \shuffle Y^{n-2})}\Bigl(G_0(X,Y)\left(\overrightarrow{10};\delta\right)\Bigr) \quad ({\rm by~(\ref{rel1})})\\
=&\Coeff_{X(Y^{m-1} \shuffle X^{n-2})}\Bigl(G_0(X,Y)(1-z;\gamma')\Bigr)+\Coeff_{Y(X^{m-1} \shuffle Y^{n-2})}\Bigl(G_0(X,Y)\left(\overrightarrow{10};\delta\right)\Bigr),
\end{align*}
we get the desired equation (\ref{Main1}).
This completes the proof of Theorem \ref{M1}.

Next,
we show Theorem \ref{M2}.
Let $\sigma \in G_K$.
%and $z$ a $K$-rational (possibly, tangential base) point on $\mathbb{P}^1 \backslash \{0,1,\infty\}$.
Since
${\mathfrak f}^{z,\gamma}_{\sigma}(X,
Y)$ is group-like in $\mathbb{Q}_{\ell} \langle \langle X,
Y \rangle \rangle$,
the shuffle relation holds for $\{\Coeff_{w}({\mathfrak f}^{z,\gamma}_{\sigma}(X,Y))\}_{w \in {\rm M}}$.
Using (\ref{lpoly1}),
(\ref{lpoly2}),
(\ref{rel_rho}),
(\ref{rel_rho2}),
(\ref{shuffle relation}) and (\ref{sh1}),
we obtain the following equalities
by making the same computations as above:
\begin{align*}
&\sum_{j=0}^{m-1}\frac{(\rho_{z,\gamma}(\sigma))^j}{j!}Li^{\ell}_{\hspace{-0.2cm}\underbrace{\scriptstyle 1,\ldots,1}_{n-2~\text{times}}\hspace{-0.2cm},m-j}(z;\gamma,\sigma)
=(-1)^{n+m-2} \cdot \Coeff_{Y(X^{m-1} \shuffle Y^{n-2})}\Bigl({\mathfrak f}^{z,\gamma}_{\sigma}(X,
Y)\Bigr),\\
&\sum_{j=0}^{n-2}\frac{(\rho_{1-z,\gamma'}(\sigma))^j}{j!}Li^{\ell}_{\hspace{-0.2cm}\underbrace{\scriptstyle 1,\ldots,1}_{m-2~\text{times}}\hspace{-0.2cm},n-j}(1-z;\gamma',\sigma)
=(-1)^{n+m-3} \cdot \Coeff_{X(Y^{m-1} \shuffle X^{n-2})}\Bigl({\mathfrak f}^{1-z,\gamma'}_{\sigma}(X,
Y)\Bigr),\\
&{\boldsymbol \zeta}^{\ell}_{\hspace{-0.2cm}\underbrace{\scriptstyle 1,\ldots,1}_{n-2~\text{times}}\hspace{-0.2cm},m}(\sigma) =(-1)^{n+m-2} \cdot \Coeff_{Y(X^{m-1} \shuffle Y^{n-2})}\left({\mathfrak f}^{\overrightarrow{10},\delta}_{\sigma}(X,Y))\right).
\end{align*}
Combining 
these equalities
and the following equality
\begin{align*}
&\Coeff_{Y(X^{m-1} \shuffle Y^{n-2})}\Bigl({\mathfrak f}^{z,\gamma}_{\sigma}(X,
Y)\Bigr)\\
=&\Coeff_{Y(X^{m-1} \shuffle Y^{n-2})}\Bigl({\mathfrak f}^{1-z,\gamma'}_{\sigma}(Y,
X)\Bigr)+\Coeff_{Y(X^{m-1} \shuffle Y^{n-2})}\left({\mathfrak f}^{\overrightarrow{10},\delta}_{\sigma}(X,Y)\right) \quad ({\rm by~(\ref{rel3})})\\
=&\Coeff_{X(Y^{m-1} \shuffle X^{n-2})}\Bigl({\mathfrak f}^{1-z,\gamma'}_{\sigma}(X,
Y)\Bigr)+\Coeff_{Y(X^{m-1} \shuffle Y^{n-2})}\left({\mathfrak f}^{\overrightarrow{10},\delta}_{\sigma}(X,Y)\right), 
\end{align*}
we get the desired equation (\ref{Main2}). 
This completes the proof of Theorem \ref{M2}.
\end{proof}

\end{document}